\documentclass{amsart}

\usepackage{amsmath}
\usepackage{amsrefs}
\usepackage{latexsym}
\usepackage{amssymb}
\usepackage{setspace}
\usepackage{xcolor}
\usepackage{ifthen}
\usepackage{dsfont}
\usepackage{graphicx, caption}
\usepackage{overpic}

\newtheorem{Thm}{Theorem}[section]
\newtheorem{Theorem}{Theorem}
\newtheorem{Lem}[Thm]{Lemma}

\newtheorem{Cor}[Thm]{Corollary}

\theoremstyle{definition}
\newtheorem{Def}[Thm]{Definition}
\newtheorem{Not}[Thm]{Notation}
\newtheorem{example}[Thm]{Example}

\theoremstyle{remark}
\newtheorem{Rem}[Thm]{Remark}

\numberwithin{equation}{section}

\newcommand{\ina}{infinite non-affine } 
\newcommand{\id}{\mathds{1}}
\newcommand{\HH}{\mathbb{H}}
\newcommand{\RR}{\mathbb{R}}
\newcommand{\NN}{\mathbb{N}}
\newcommand{\Iso}{\text{\textit{Iso}}}
\makeatletter
\@namedef{subjclassname@2020}{\textup{2020} Mathematics Subject Classification}
\makeatother

\newboolean{ComsOn}
\setboolean{ComsOn}{false}
\ifthenelse{\boolean{ComsOn}}{\newcommand{\com}[2][blue]{\textcolor{#1}{#2}}}{\newcommand{\com}[2][blue]{}}

\renewcommand{\emptyset}{\varnothing}

\begin{document}

\title{Reflection length at infinity in hyperbolic reflection groups}
\author{Marco Lotz}
\address{Fakultät für Mathematik, Otto-von-Guericke-University Magdeburg, 
Universitätsplatz 2,
39106 Magdeburg}
\curraddr{Institut für Algebra und Geometrie (IAG),
Otto-von-Guericke-University Magdeburg, 
Universitätsplatz 2,
39106 Magdeburg}
\email{marco.lotz@ovgu.de}



\date{March 15, 2023.}



\begin{abstract}
In a discrete group generated by hyperplane reflections in the $n$-dimensional hyperbolic space, the reflection length of an element is the minimal number of hyperplane reflections in the group that suffices to factor the element. For a Coxeter group that arises in this way and does not split into a direct product of spherical and affine reflection groups, the reflection length is unbounded. The action of the Coxeter group induces a tessellation of the hyperbolic space. After fixing a fundamental domain, there exists a bijection between the tiles and the group elements. We describe certain points in the visual boundary of the $n$-dimensional hyperbolic space for which every neighbourhood contains tiles of every reflection length. To prove this, we show that two disjoint hyperplanes in the $n$-dimensional hyperbolic space without common boundary points have a unique common perpendicular.  
\end{abstract}

\maketitle

\section*{Introduction}

Two distinct hyperplanes in the $n$-dimensional hyperbolic space $\HH^n$ intersect, have a common point in their visual boundary $\partial \HH^n$ or are ultra-parallel. A finite set $S$ of reflections across hyperplanes generates a Coxeter group $W$ if pairs of the corresponding hyperplanes don't intersect or have a submultiple of $\pi$ as a dihedral angle.
The action of $W$ as a discrete subgroup of the group of isometries of $\HH^n$ is proper. A strict fundamental domain $P$ for this action is an intersection of half-spaces of the hyperplanes corresponding to the generators. This induces a tessellation of $\HH^n$ with copies of $P$. There exists a bijection between the tiles and the elements in $W$.
The minimal number of hyperplane reflections in $W$ that suffices to reflect a tile $wP$ with $w\in W$ back to the fundamental domain $P$ is its \textit{reflection length}. The set of hyperplane reflections $R$ in $W$ is exactly the set of the elements that are conjugated to a generator in $S$. The reflection length of the tile $wP$ is equal to the word length $l_R(w)$ of the element $w$ with respect to the generating set $R$, which we call reflection length.\par
The reflection length in spherical as well as in affine reflection groups is well understood, formulas exist and it is a bounded function on these groups (see \cite{Carter1972} and \cite{Lewis2018}). On the other hand, little is known about reflection length in Coxeter groups that don't split into a direct product of spherical and affine reflection groups. We call these groups \textit{\ina}. Duszenko proved that the reflection length is unbounded in \ina Coxeter groups (see \cite{Duszenko2011}).  For groups without braid relations between at least three distinct generators, powers of Coxeter elements have unbounded reflection length (see \cite{Drake2021}). In general, this does not hold and it is hard to find elements with large reflection length.\par
We call \ina Coxeter groups that are generated by finitely many hyperplane reflections in $\HH^n$ \textit{hyperbolic reflection groups} and denote them often as a tuple with their generating set.
Let $\overline{\HH}{}^n$ be the union of $\HH^n$ with its visual boundary $\partial \HH^n$ equipped with the cone topology.
In this work, it is shown that tiles with reflection length $n$ exist in every neighbourhood of certain points in the visual boundary. These points are namely common ideal points of two distinct hyperplanes corresponding to reflections in $R$ or endpoints of the common perpendicular of two ultra-parallel hyperplanes.

\begin{Theorem}\label{Thm: Reflection length n in certain neighbourhood}
Let $(W,S)$ be a hyperbolic reflection group in $\HH^n$ with fundamental domain $P$. Let $R$ be the set of reflections in $W$. Let $U$ be a neighbourhood in $\overline{\HH}{}^n$ of a point $\xi$ in $\partial \HH^n$. Suppose $\xi$ satisfies one of the following conditions:
\begin{itemize}
\item[(i)] $\xi$ is a common point of two parallel hyperplanes $H_r, H_{r'}$ with $r,r'\in R$.
\item[(ii)]$\xi$ is an endpoint of the common perpendicular of two ultra-parallel hyperplanes $H_r, H_{r'}$ with $r,r'\in R$.
\end{itemize}
For every $k\in \NN$ there exists $w\in W$ with $l_R(w)=k$ such that the domain $wP$ is contained in $U$.
\end{Theorem}

Here, the subscription of the hyperplanes with reflections means that the reflection across this hyperplane is an element in $W$. To consider these two types of ideal points, we prove that every pair of ultra-parallel hyperplanes in $\HH^n$ has a unique common perpendicular. This generalizes a theorem of Hilbert in $\HH^2$ for arbitrary dimension. 

\begin{Theorem}[Ultra-parallel theorem for subspaces]\label{Thm: Ultra-parallel theorem for subspaces}
Every pair of ultra-parallel geodesic subspaces in $\mathbb{H}^n$ has a common perpendicular. If both these subspaces are hyperplanes, the common perpendicular is unique. Every hyperplane intersecting both hyperplanes in a right angle contains this perpendicular.
\end{Theorem}

In a final step, we restrict our attention to hyperbolic reflection groups with a convex polytope as a strict fundamental domain. We show that the set of ideal points, which satisfy condition (i) or (ii) in  Theorem \ref{Thm: Reflection length n in certain neighbourhood}, is dense in $\partial \HH^n$. As a consequence, Theorem \ref{Thm: Reflection length n in certain neighbourhood} still holds in this setting if we drop the conditions for the ideal point.\par
The proofs of Theorem 1 and 2 are of geometric nature and use the Poincaré ball model as well as the hyperboloid model for $\HH^n$. \newline
\paragraph{\bf{Structure}} This article has five sections. After recalling briefly the foundations of the hyperbolic space and the connection between Coxeter groups and hyperbolic geometry in the preliminaries, we prove the Ultra-parallel theorem for subspaces with fundamental Euclidean geometry in Section \ref{Section: Ultra-parallel theorem for subspaces}. In the third section, the main result, Theorem \ref{Thm: Reflection length n in certain neighbourhood}, is established. Afterwards, we prove some geometric lemmata in Section \ref{Section 4} in preparation for the last section, where we consider only convex polytopes as strict fundamental domains.\newline
\com{ mention that not all coxeter groups in general touched/comparison with affine groups}

\section{Preliminaries}\label{section 1}

In this section, we recall the basic notions of the real hyperbolic $n$-space $\HH^n$, foundations about Coxeter groups as well as the connection between hyperbolic geometry and theory of Coxeter groups. $\HH^n$ is the complete, simply connected Riemannian $n$-dimensional manifold of constant sectional curvature $-1$. The hyperboloid model and the Poincaré ball model for $\HH^n$ are introduced. Moreover, we define the visual boundary and the cone topology on $\overline{\HH}{}^n$.
For a thorough treatment of the topics we refer the reader to \cite{Bridson2009}, \cite{Davis2012} and \cite{Vinberg1993}.

\begin{Def}
Let $\mathbb{X}$ be a metric space and $I$ an interval of the form $(-\infty,\infty),$ $[a, \infty)$ or $[a,b]$ with $a,b\in \RR$.  A \textit{geodesic line}, \textit{geodesic ray}, \textit{geodesic segment}  in $\mathbb{X}$ is an isometry $\gamma: I \to \mathbb{X}$. We write $\gamma\subseteq \mathbb{X}$ instead of $\gamma(I)\subseteq \mathbb{X}$.
The space $\mathbb{X}$ is called \textit{geodesic space} if for every pair of points $x,y$ in $\mathbb{X}$ there exists a geodesic segment $\lambda: [a,b]\to \mathbb{X}$ with $\lambda(a)=x$ and $\lambda(b)=y$.
\end{Def}


\paragraph{\bf{The hyperboloid model}} Let $\RR^{n,1}$ be the real vector space equipped with the 
symmetric bilinear form $\langle \cdot |\cdot \rangle$ of type $(n, 1)$ ($n$ positive and $1$ negative eigenvalue). 
The $n$-dimensional \textit{hyperboloid model} is the upper sheet of a hyperboloid defined as follows:
\[
\HH^n:=\{v=(v_1,\dots , v_{n+1})\in \mathbb{R}^{n,1}\mid \langle v|v\rangle=-1, v_{n+1}>0\}.
\]
The bilinear form induces a metric $d(x,y)=\cosh^{-1}(-\langle x|y\rangle)$ for $x,y\in\HH^n$. We always consider $\HH^n$ as a metric space.
Non-empty intersections of $2$-dimensional subspaces of $\RR^{n,1}$ with $\HH^n$ are the geodesic lines in $\HH^n$. For two geodesic lines $\gamma_1$ and $\gamma_2$ intersecting in a point $p\in \HH^n$ there exist unit vectors $u_{i}\in \RR^{n,1}$ with $\langle u_{i}| p \rangle=0$ and such that $u_{i}$ is contained in the $2$-dimensional subspace of $\RR^{n,1}$ according to $\gamma_i$. The hyperbolic angle between $\gamma_1$ and $\gamma_2$ is the unique number $\alpha\in [0,\pi]$ with $\alpha= \cos^{-1}(\langle u_{1}|u_{2}\rangle)$.\par
Hyperplanes are defined as non-empty intersections of $n$-dimensional subspaces of $\RR^{n,1}$ with $\HH^n$ and are isometric to $\HH^{n-1}$. For a hyperplane $H$, we denote the corresponding subspace in $\RR^{n,1}$ with $V_{H}$. We have two (closed) half-spaces $V_H^+, V_H^-\subseteq R^{n,1}$ with $V_H^+\cup V_H^-= \RR^{n,1}$ and $H^\pm=V_H^\pm\cap \HH^n$ with $H_i^+\cup H_i^-= \HH^n$. Every hyperplane $H$ induces an isometry on $\HH^n$, the reflection $s$ across $H$:
\begin{equation}\label{Eq: Formular reflection on hyperplane}
s:\HH^n\to \HH^n\, ; \qquad x\mapsto x-2\langle u_{H},x\rangle u_{H}
\end{equation}
where $u_{H}$ is the up to sign unique unit vector orthogonal to $V_{H}$ in $\RR^{n,1}$ with respect to the bilinear from $\langle\cdot\,,\cdot\rangle$. The fixed point set of $s$ is exactly $H$ and one half-space is mapped to the other one.\par
\vspace{10pt}

\paragraph{\bf{The Poincaré ball model}} In the \textit{Poincaré ball model}, the points of the hyperbolic space are represented by points in the open unit ball $D^n$ in the Euclidean $n$-space $\mathbb{E}^n$. There exists a homeomorphism $D^n\to \HH^n$ by which the metric on $\HH^n$ can be pulled back to $D^n$.
Geodesic lines are the lines and circles n $\mathbb{E}^n$ that meet the boundary $\partial D^n=\mathbb{S}^{n-1}$ in a right angle. The angle between two segments issuing from a point is equal to the Euclidean angle between the segments. Let  $\widehat{\mathbb{E}}^n:= \mathbb{E}^n\cup \{\infty\}$ denote the one-point compactification of $\mathbb{E}^n$.
The hyperplanes in this model are the intersections of $D^n$ with $(n-1)$-dimensional spheres in $\widehat{\mathbb{E}}^n$  that intersect $\mathbb{S}^{n-1}$  orthogonally. If a hyperplane $H$ is the intersection of a $(n-1)$-dimensional sphere containing $\infty$ with $D^n$, the reflection on $H$ is the reflection on $H$ as a $(n-1)$-dimensional subspace  in $\mathbb{E}^n$ restricted to $D^n$. Else $H$ is represented by a sphere $S$ with radius $r$ and centre $c$. In this case, the reflection across $H$ is $i_S$,  the inversion on $S$ in the one-point compactification $\widehat{\mathbb{E}}^n$ restricted to $D^n$:
\begin{equation}\label{Eq: Formula inversion on a sphere}
i_S:\widehat{\mathbb{E}}^n\to \widehat{\mathbb{E}}^n\, ; \qquad x\mapsto \dfrac{r^2}{||x-c||^2}\cdot (x-c)+c.
\end{equation}
\par 

\begin{Def}$ $\newline
\vspace{-10pt}

\begin{enumerate}
\item A \textit{convex polyhedral cone} $C$ in $\RR^{n,1}$ is the intersection of a finite number of
half-spaces marked by $n$-dimensional subspaces in $\RR^{n,1}$. 
\item A \textit{convex polytope} (with possibly ideal vertices) in $\HH^n$ is the intersection of $\HH^n$ with a convex polyhedral cone $C$ in $\mathbb{R}^{n,1}$ that is contained in the light cone \[\{ v=(v_1,\dots , v_{n+1})\in \mathbb{R}^{n,1}\;|\; \langle v|v\rangle\leq 0, v_{n+1}>0 \}.\] \label{Def: Convex polytope} 
\item  A \textit{convex polyhedron} in $\HH^n$ is an intersection of finitely many half-spaces in $\HH^n$ having a non-empty interior. 
\end{enumerate}
The notion of a convex polytope 
also includes convex polyhedra of finite volume.
\end{Def}

\subsection{Coxeter groups and hyperbolic reflection groups}

We look at discrete subgroups of $\Iso(\HH^n)$, the group of isometries of $\HH^n$, generated by finitely many hyperplane reflections. Before restricting to hyperbolic geometry, Coxeter groups are introduced algebraically, which can be understood as abstract reflection groups. 
\begin{Def}
A group with a finite set of generators $S=\{s_1, \dots s_n\}$ that admits a presentation of the form
\[
\langle S\;|\; s_i^2= (s_is_j)^{m_{ij}}= \id \, ,\,\forall i,j\in \{1,\dots, n\},\; m_{ij}\in \mathbb{N}_{\geq 2} \cup \{\infty\} \rangle,
\]
is called \textit{Coxeter group}.  We call the pair $(W,S)$ \textit{Coxeter system}.\par
An \textit{\ina} Coxeter group is a Coxeter group that does not split into a direct product of spherical and affine reflection groups. \par
A \textit{special subgroup} of a Coxeter system $(W, S)$ is a subgroup generated by a subset $S'\subseteq S$. 
\end{Def}

Spherical and affine reflection groups are completely classified (e.g. see \cite{Humphreys1990}). We relate \ina Coxeter groups to reflection groups in $\HH^n$. 
Classical works on Coxeter groups and reflection groups are \cite{Humphreys1990} and \cite{Davis2012}.

\begin{figure}
\includegraphics[scale=0.4]{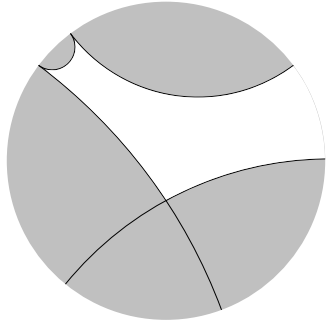}
\caption{A convex polyhedron in the Poincaré ball model marked white}
\end{figure}

\begin{Def}
A convex polyhedron
\[
P=\bigcap_{i\in I} H_i^{\varepsilon_i}
\]
with $\varepsilon_i\in \{+,-\}$ is called a \textit{Coxeter polyhedron} if for all intersecting hyperplanes $H_i, H_j$  with  $i,j\in I, \;i\neq j$ the dihedral angle $\measuredangle(H_i, H_j)$ is a
submultiple of $\pi$. We call a hyperplane $H_i$ \textit{wall} if the closure $\overline{P}$ and $H_i$ intersect.
\end{Def}

\begin{Def}
Let $G$ be a discrete group. A $G$-action on a Hausdorff
space $Y$ is \textit{proper} (or \textit{properly discontinuous}) if the following three
conditions are satisfied:
\begin{itemize}
\item[(i)] $Y/G$ is Hausdorff.
\item[(ii)] For each $y \in  Y$, the stabilizer $G_y= \{g \in G \;|\; gy = y\}$ is finite.
\item[(iii)] Each $y \in Y$ has a $G_y$-stable neighbourhood $U_y$ such that $gU_y \cap  U_y = \emptyset$
for all $g \in G - G_y$.
\end{itemize}
\end{Def}

\begin{Def}
Suppose a group $G$ acts on a space $Y$. A closed subset $A \subseteq Y$ is
a \textit{strict fundamental domain} if each $G$-orbit intersects $A$ in exactly one point.
\end{Def}

The following theorem is the most important connection between hyperbolic geometry and Coxeter groups for this work. 

\begin{Thm}[Vinberg, \cite{Vinberg1993}, pp. 199]\label{Thm: Polyhedron reflection group is Coxeter group, Vinberg+Davis}
Let $P=\bigcap_{i\in I} H_i^{\varepsilon_i}$ be a Coxeter polyhedron in $\HH^n$ and $W(P)$
the group generated by reflections $\{s_i\mid i\in I\}$ on its walls. 
Under this assumptions, the following holds: 
\begin{itemize}
\item[(i)] $W(P)$ is a discrete subgroup of $Iso(\HH^n)$ generated by hyperplane reflections.  
\item[(ii)] $W(P)$ acts properly on $\HH^n$.
\item[(iii)] $P$ is a strict fundamental domain for the $W(P)$-action on $\HH^n$.
\item[(iv)] $W(P)$ is a Coxeter group with defining relations $s_i^2=id$ for all $i\in I$ and $(s_js_k)^{m_{jk}}=id$ for intersecting $H_j$ and $H_k$ with dihedral angle $\measuredangle(H_j, H_k)=\frac{\pi}{m_{jk}}$.
\item[(v)] The stabilizer $W(P)_x$ of any point $x\in P$ (including points at infinity) is generated by reflections across the walls of $P$ containing $x$.
\end{itemize}
\end{Thm}

\begin{Cor}
The fundamental domains $wP$ with $w\in W(P)$ cover the space $\HH^n$ and there is a bijection between $W(P)$ and the set of fundamental domains $wP$. Namely,
\[
w\mapsto wP\quad \text{with}\; w\in W(P).
\]
\end{Cor}

\begin{Def}
We call a discrete subgroup of $Iso(\HH^n)$ generated by hyperplane reflections \textit{hyperbolic reflection group} if the underlying Coxeter system is \ina. In the following, hyperbolic reflection groups will be abbreviated often with the underlying Coxeter system $(W,S)$, where the set $S$ is the set of hyperplane reflections $s_i$ corresponding to a finite set of hyperplanes $\{H_1, \dots H_n\}$ in $\HH^n$. 
\end{Def}

A good general reference for hyperbolic reflection groups is chapter $5$ in \cite{Vinberg1993}. 

\begin{Rem}
Isomorphic groups can be generated by different hyperplane arrangements in $\HH^n$. Some Coxeter groups cannot be embedded as discrete subgroups generated by hyperplane reflections in $Iso(\HH^n)$ for any $n\in \NN$ (see \cite{Felikson2005}). Moreover, hyperbolic Coxeter polytopes and polyhedra aren't classified completely yet. To our knowledge, there are no general criteria for Coxeter groups that specify when the group is isomorphic to a hyperbolic reflection group. Hence, we consider hyperbolic reflection groups and not Coxeter groups with additional properties. Compact hyperbolic Coxeter polytopes do not exist in dimensions higher than $29$ (see \cite{Vinberg1981}). Finite volume hyperbolic Coxeter polytopes do not exist in dimensions higher than $995$ (see \cite{Khovanskii1986} and \cite{Prokhorov1987}). 
\com{Examples for first sentence}
\end{Rem}

Before we give examples, we formulate a theorem from \cite{Duszenko2011}, which combines a list of results of chapter V in \cite{Bourbaki2007}. It states that minimal \ina Coxeter groups can be represented as hyperbolic reflection groups that have a simplex as a fundamental domain. 

\begin{Def}
A \textit{minimal \ina}Coxeter group is a \ina Coxeter group that only has direct products of spherical and affine reflection groups as proper special subgroups. 
\end{Def}
\begin{Rem}
Every \ina
Coxeter group has a minimal \ina special subgroup; it is any special
subgroup minimal with respect to inclusion among the \ina ones.
\end{Rem}

\begin{Thm}[Duszenko, \cite{Duszenko2011}]\label{Thm: minimal non-affine}
Every minimal \ina
Coxeter system $(W, S)$ can be faithfully represented as a discrete reflection group
acting on the hyperbolic space $\HH^n$, where $n=|S|-1$. The elements of $S$ act as
reflections with respect to codimension $1$ faces of a simplex $\Delta$ (some of the vertices
of $\Delta$ might be ideal).
\end{Thm}

\begin{example}
According to Theorem \ref{Thm: minimal non-affine}, there exists a hyperbolic reflection group with a simplex as a fundamental domain isomorphic to $(W,S)$ for every minimal \ina Coxeter system $(W,S)$. The tessellations of $\HH^2$ in the Poincaré ball model in Figure \ref{Fig: Tessellation} correspond to hyperbolic reflection groups, which are isomorph to minimal \ina Coxeter systems.  
\begin{figure}
\begin{minipage}[h]{0.45\textwidth}
\includegraphics[scale=0.5]{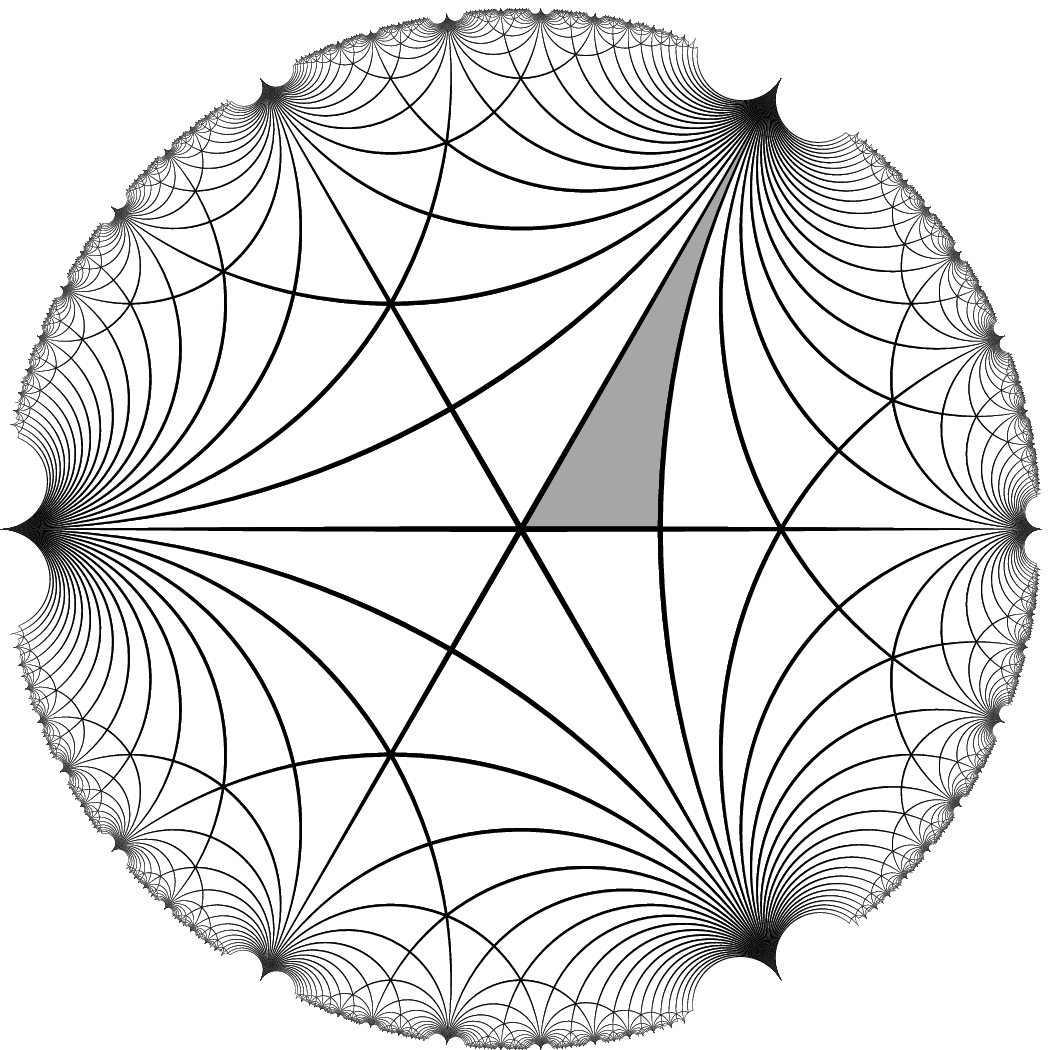}
\end{minipage}
\hspace{10pt}
\begin{minipage}[h]{0.45\textwidth}
\includegraphics[scale=0.5]{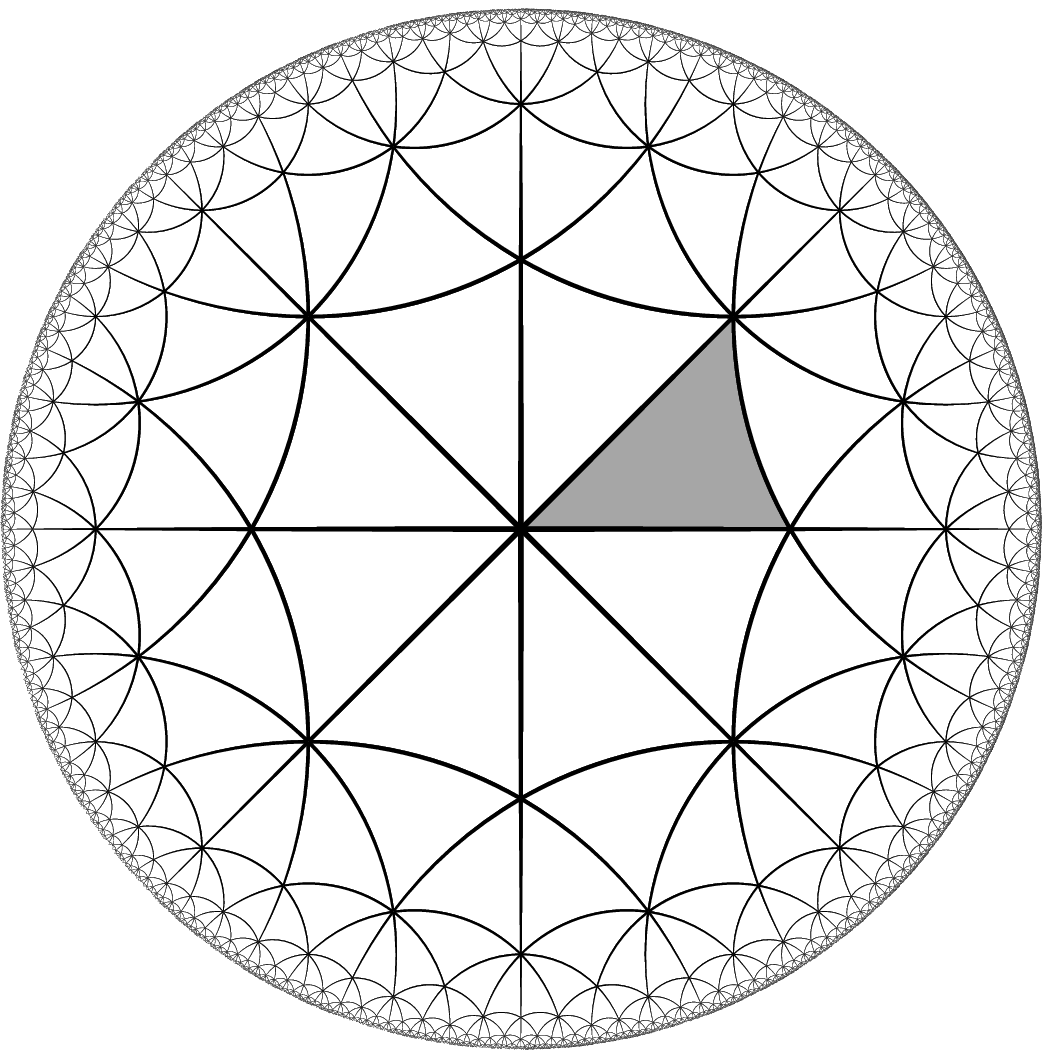}
\end{minipage}
\caption{Tessellations in the Poincaré ball model, fundamental domains marked grey}
\label{Fig: Tessellation}
\end{figure}
\vspace*{7pt}

The Coxeter system corresponding to the left figure has the presentation
\[
\langle s_1, s_2, s_3\mid s_i^2= (s_1s_2)^2= (s_1s_3)^3= (s_1s_3)^\infty= \id \rangle
\]
and the system corresponding to the right figure
\[
\langle s_1, s_2, s_3\mid s_i^2= (s_1s_2)^3= (s_1s_3)^4= (s_1s_3)^4= \id \rangle.
\]
\end{example}

\subsection{Abstract reflections}
From now on, let $(W,S)$ be a Coxeter system.
\begin{Def}
The conjugates of the standard generators in $S$ are called \textit{reflections}. A reflection is an involution and the set of reflections \[R:=\{wsw^{-1}\;|\; w\in W, \; s\in S\}\] is a generating set for the group $W$. Thus, analogously to the word length, we can define the \textit{reflection length} 
\[
\l_R: W\to\mathbb{N}\, ; \qquad w\mapsto \min \{n\in \NN\mid w\in R^n\}
\] 
with $R^n= \{ r_1\cdots r_n\in W\mid r_i\in R\}$. The identity element $\id$ has reflection length zero. We call minimal reflection factorisations of an element \textit{$R$-reduced} and minimal standard generator factorisations \textit{$S$-reduced}.
\end{Def}

The following theorem is the only way known to compute the reflection length in an arbitrary Coxeter group. Roughly speaking, it shows that the reflection length of an element can be understood as a measure of how many generators an element differs from the identity in the group.

\begin{Thm}[Dyer, \cite{Dyer2001}]\label{Thm: Dyer's Theorem}
Let $w = s_1\cdots s_n$ be a $S$-reduced expression for $w\in W$. Then $l_R(w)$ is equal to
the minimum of the natural numbers $p$ for which there exist $1\leq i_1 < \cdots < i_p \leq n$ such that $s_1 \cdots \hat{s_{i_1}}\cdots \hat{s_{i_p}}\cdots s_n= \id$, where a hat over a factor indicates its omission.
\end{Thm}

\begin{Rem}
In spherical and affine reflection groups there exist formulas for the reflection length and $l_R$ is a bounded function (see \cite{Carter1972} and \cite{Lewis2018}). The reflection length is additive on direct products (see \cite{Mccam2011}, Proposition 1.2). In contrast to this, \ina Coxeter groups have unbounded reflection length.  
\end{Rem}

\begin{Thm}[Duszenko, \cite{Duszenko2011}]\label{Thm: Duszenko LR unbpunded}
Let $(W,S)$ be an \ina Coxeter system. The reflection length $l_R$ is an unbounded function on $W$.
\end{Thm}

\begin{Rem}
Let $(W,S)$ be a hyperbolic reflection group in $\HH^n$ with fundamental polyhedron $P$, walls $\{H_1, \dots, H_m\}$ and generating set $S=\{s_1, \dots, s_m\}$.  The reflection $ws_iw^{-1}\in R$ ($w\in W$, $1\leq i\leq m$) acts on $\HH^n$ as the hyperplane reflection across $wH_i$. For each polyhedron $vP$ with $v\in W$, the minimal number of hyperplane reflections across hyperplanes in $\{wH_i\subseteq \HH^n\mid w\in W,1\leq i\leq m\}$ that suffices to reflect $vP$ onto $P$ is exactly the reflection length of $v$. For a reflection $r\in R$, we denote the corresponding hyperplane in $\HH^n$ with $H_r$. 
\end{Rem}

\subsection{Hyperbolic hyperplanes and their boundary}\label{Subsection: Hyperbolic hyperplanes and their boundary}

\begin{Def}
Two geodesic rays $\gamma, \gamma': [0,\infty)\to \mathbb{X}$ in a metric space $\mathbb{X}$ are called \textit{asymptotic} if $\sup \{d(\gamma (x), \gamma'(x))\;|\; x\in [0,\infty) \}<\infty$. This is an equivalence relation on the set of geodesic rays.\par
The \textit{visual boundary} $\partial\mathbb{X}$ 
is the the set of equivalence classes of geodesic rays. We call the elements \textit{ideal points} and denote them with $\gamma(\infty)$. The union of the space and its visual boundary is denoted $\overline{\mathbb{X}}:=\mathbb{X} \cup \partial \mathbb{X} $.
\end{Def}

\begin{Rem}
Two geodesic rays are asymptotic if and only if their Hausdorff distance is finite.
In our setting, the visual boundary is the same as the Gromov boundary because $\HH^n$ is a CAT$(0)$ visibility space.
\end{Rem}

\begin{example}
In the Poincaré ball model, the visual boundary $\partial \mathbb{H}^n$ is exactly the unit sphere $\mathbb{S}^{n-1}$. In the hyperboloid model, if two geodesic rays are asymptotic, then the intersection of the corresponding subspaces in $\RR^{n,1}$ is a line contained in the boundary of the light cone $\{ v=(v_1,\dots , v_{n+1})\in \mathbb{R}^{n,1}\;|\; \langle v|v\rangle = 0, v_{n+1}>0 \}$ (cf. \cite{Bridson2009}, pp. 262). 
\end{example}

We equip $\overline{\mathbb{X}}$ with the cone topology.

\begin{Def}\label{Def: cone topology}
Fix a point $x_0$ in a CAT$(0)$ metric space $(\mathbb{X},d)$. Let $c$ be a geodesic ray with $c(0)=x_0$ and $r, \epsilon>0$ be real numbers. Let $p_r$ be the projection of $\overline{\mathbb{X}}$ onto the closed ball $\overline{B}(c(0), r)$: For $x\notin B(c(0), r)$ the projection $p_r(x) $ is the point in the segment $[x_0, x]$, geodesic ray respectively, with distance $r$ from $x_0$.
We define 
\[
U(c, r, \epsilon):= \{ x\in \overline{\mathbb{X}}\;|\; d(x, c(0))>r, \; d(p_r(x), c(r))<\epsilon\}.
\]
The sets $U(c, r, \epsilon)$ form a neighbourhood basis for $c(\infty)$ and the set of all open balls $B(x, r)$ together with all sets of the
form $U(c, r, \epsilon)$, where c is a geodesic ray with $c(0)= x_0$, is a basis of the \textit{cone topology} on $\overline{\mathbb{X}}$. The cone topology is independent of the choice of the point $x_0$. 
\end{Def}

Now, we turn to geodesic subspaces of $\HH^n$ of arbitrary dimension in the hyperboloid model. Geodesic lines in $\HH^n$ are exactly the non-empty intersections of $2$-dimensional vector subspaces of $\RR^{n,1}$ with $\HH^n$. The $m$-dimensional geodesic subspaces of $\HH^n$ are the non-empty intersections of $(m+1)$-dimensional vector subspaces of $\RR^{n,1}$ with $\HH^n$ ($m\leq n$). These subspaces are isometric to $\HH^m$.\par
For a subspace $H\subseteq \HH^n$ we denote the corresponding subspace in $\RR^{n,1}$ with $V_{H}$.

\begin{Def}
Two 
disjoint arbitrary dimensional geodesic subspaces $H_1$ and $H_2$ of $\HH^n$  are called \textit{ultra-parallel} if there exist no geodesic rays $\gamma_1\subseteq H_1$ and $\gamma_2\subseteq H_2$ that are asymptotic. In other words, the subspaces do not have a common point in $\partial\HH^n$. 
\end{Def}
\begin{Def}\label{Definition: intersecting in a right angle}
Two intersecting geodesic subspaces $H_1, H_2\subseteq \HH^n$ are \textit{intersecting in a right angle} if their intersection is non-empty in $\HH^n$, and there exist non-trivial vectors $u_1, \dots, u_m$ with $ m>0$ in the orthogonal complement $V^\bot_{H_1}\subseteq \RR^{n,1}$ such that  
\[
\langle (V_{H_1}\cap V_{H_2})\cup\{u_1,\dots,u_m\} \rangle=V_{H_2}.
\]

A \textit{perpendicular} of a geodesic subspace $H\subseteq\HH^n$ is a geodesic line that intersects $H$ in a right angle as a subspace.
\end{Def}

\begin{Rem}
The definition above is symmetric. By definition $V_{H_1}$ is contained in $\langle\{ u_1,\dots,u_m\}\rangle^\bot$ and we have $V^\bot_{H_2}= \langle\{ u_1,\dots,u_m\}\rangle^\bot \cap (V_{H_1}\cap V_{H_2})^\bot$. So we can select vectors $\{v_1, \dots, v_n\}\subseteq V^\bot_{H_2}$ such that 
\[
\langle (V_{H_1}\cap V_{H_2})\cup\{v_1, \dots, v_n\} \rangle=V_{H_1}.
\]
\end{Rem}

\begin{example}
Let $H_1$ be a hyperplane and $H_2$ be a geodesic line with $H_1\cap H_2\neq \emptyset$ and $H_2\nsubseteq H_1$. The subspaces $V_{H_1}\cap V_{H_2}$ and $V_{H_1}^\bot$ are one-dimensional. $H_2$ is a perpendicular of $H_1$ if and only if $V_{H_1}^\bot\subseteq V_{H_2}$. \par
Two hyperplanes are intersecting in a right angle if and only if the corresponding orthogonal unit vectors are orthogonal.

\end{example}

\section{Ultra-parallel theorem for subspaces in $\mathbb{H}^n$}\label{Section: Ultra-parallel theorem for subspaces}
David Hilbert proved the following theorem based on his system of axioms for $\mathbb{H}^2$. 

\begin{Thm}[Hilbert, \cite{Hilbert1913} p. 149]
Any two ultra-parallel geodesic lines in $\mathbb{H}^{2}$ have a common perpendicular.
\end{Thm}

We extend this theorem to ultra-parallel geodesic subspaces in $\mathbb{H}^n$ and prove the uniqueness of the common perpendicular in case both subspaces are hyperplanes. Our proof is different from Hilbert's proof and includes his theorem. Therefore, we need the following geometrical lemma, which follows from basic Euclidean geometry.

\begin{Lem}\label{Lem: double intersection}
Let $\mathbb{S}^n$ be the unit Sphere embedded in the $(n+1)$-dimensional Euclidean space $\mathbb{E}^{n+1}$ and $S_1$, $S_2$ be two spheres of dimension $n$ or lower intersecting $\mathbb{S}^n$ orthogonally. 
\begin{itemize}
\item[(i)] If $S_1$ and $S_2$ are not intersecting in the closed unit ball $\overline{D}{}^{n+1}$, the segment $[C_1, C_2]$ between the centres $C_1$ and $C_2$ of $S_1$ and $S_2$ intersects $\mathbb{S}^n$ exactly two times.
\item[(ii)] If $S_1$ and $S_2$ are not intersecting in the open unit ball $D^{n+1}$, $S_1\cap S_2$ is a point in $\mathbb{S}^{n}$ or empty. 
\end{itemize}
\end{Lem}

In the Poincaré-ball model for $\HH^n$ hyperplanes are represented by $(n-1)$-dimensional spheres that intersect $\mathbb{S}^{n-1}$ orthogonally. This is why the following corollary is immanent.
\begin{Cor}\label{Lem: Intersection of Hyperplanes empty or point}
Let $H_1$ and $H_2$ be two parallel hyperplanes in $\HH^n$. The intersection of $\partial H_1$ and $\partial H_2$ is empty or a point in $\partial \HH^n$.
\end{Cor}

\setcounter{Theorem}{1}
\begin{Theorem}[Ultra-parallel theorem for subspaces]\label{Thm: Ultra-parallel theorem for subspaces}
Every pair of ultra-parallel geodesic subspaces in $\mathbb{H}^n$ has a common perpendicular. If both these subspaces are hyperplanes, the common perpendicular is unique. Every hyperplane intersecting both hyperplanes in a right angle contains this perpendicular. 
\end{Theorem}
\begin{proof}
The proof consists of two parts. The existence part is shown constructively in the Poincaré ball model and the uniqueness part follows in the upper hyperboloid model.\par 
Let $H_a$ and $H_b$ be two ultra-parallel geodesic subspaces in $\mathbb{H}^n$. In the Poincaré ball model these are represented by spheres $S_a$ and $S_b$ that intersect $\mathbb{S}^{n-1}$ orthogonally. 
Since they are ultra-parallel, $S_a$ and $S_b$ do not intersect in the closed unit ball $\overline{D}{}^n$. 
Without loss of generality, we can assume that there exists a hyperplane $S$ through the centre $M$ of $\mathbb{S}^{n-1}$ such that we have $S_a\subseteq \mathring{S}^+$ and $S_b\subseteq \mathring{S}^-$ for the open half-spaces.  If necessary, apply a translation on the Poincaré ball model. Translations are isometries and preserve Euclidean angles (see \cite{Bridson2009}, 6.5 and 6.11). 
 We proceed with the following construction: \par
 Let $ M_a,M_a$ be the centres of the spheres $S_a$ and $S_b$. Since we assumed that $S_a$ and $S_b$ can be separated by a hyperplane through $M$, neither $M_a$ nor $M_b$ is $\infty$.  The segment $[M_a,M_b]$ in $\mathbb{E}^n$ intersects $\mathbb{S}^{n-1}$ in two points $I_a$ and $I_b$, because $S_a$ and $S_b$ are ultra-parallel (see Lemma \ref{Lem: double intersection}). In Addition, it intersects $S_a$ and $S_b$ orthogonally, since the segment issues from the centres of these spheres. If $M\in [M_a,M_b]$, the common perpendicular in $\overline{\HH}{}^n$ is represented by  $[I_a,I_b]$.\par
In case $M\notin [M_a,M_b]$, the triangle $\Delta MI_aI_b$ is an isosceles triangle and the segments $[M,I_a]$ and $[M,I_b]$ are contained in a unique plane $E\subseteq \mathbb{E}^n$.
Let $X$ be the middle point of the segment $[I_a,I_b]$.
Since the angles $\measuredangle I_b I_a M$ and $\measuredangle M I_b I_a$ are equal, the unique tangent lines to $\mathbb{S}^{n-1}$ in $E$ at $I_a$ and $I_b$ as well as the line $(M,X)$   intersect all in a point $Q$. We can construct a circle $C_Q$ with centre $Q$ through $I_a$ and $I_b$, because the triangles $\Delta I_aXQ$ and $\Delta I_bXQ$ are congruent.
$C_Q$ intersects $\mathbb{S}^{n-1}$ orthogonally, since its centre is on a tangent line.
Thus, it represents a geodesic line in the Poincaré ball model. 
It maintains to show that $C_Q$ intersects $S_a$ and $S_b$ orthogonally.
The following equations are implied by the Pythagorean theorem, where $r_M, r_a, r_b$ and $r_Q$ are the radii of the corresponding spheres in $\mathbb{E}^n$: 
\begin{figure}
\begin{overpic}[scale=0.3]
{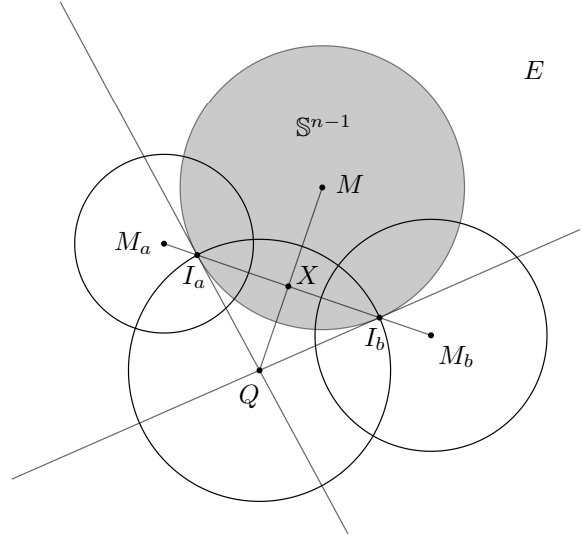}
\put(50,70){$\mathbb{S}^{n-1}$}
\put(90,80){$E$}
\put(57,60){$M$}
\put(18,50){$M_a$}
\put(75,30){$M_b$}
\put(50,44){$X$}
\put(40,23){$Q$}
\put(30,44){$I_a$}
\put(62,33){$I_b$}
\end{overpic}
\caption{Cross section by $E$}
\label{Fig:Ultra pralallel theorem}
\end{figure}
\begin{align}
\tag{i}   r_M^2 &= d(I_a, X)^2+d(M, X)^2 \label{i}\\
\tag{ii}   r_Q^2 &=d(I_a, X)^2+d(Q,X)^2 \label{ii}\\
\tag{iii}   d(M,M_a)^2 &= d(M_a, X)^2+d(M,X)^2\label{iii}\\
\tag{iv}   d(M_a, Q)^2 &= d(Q,X)^2+d(M_a,X)^2\label{iv}\\
\tag{v}   d(M, M_a)^2 &=r_M^2+r_a^2 \label{v}
\end{align}
By substituting we get: 
\begin{align*}
r_a^2+ r_Q^2 &\overset{\text{(\ref{v})}}{=}d(M, M_a)^2-r_M^2+r_Q^2 \overset{\text{(\ref{iii})}}{=} d(M_a, X)^2+d(X,M)^2-r_M^2+r_Q^2\\
 &\overset{\text{(\ref{ii})}}{=}d(M_a, X)^2 +d(M, X)^2-r_M^2+d(I_a, X)^2+d(Q, X)^2\\
  & \overset{\text{(\ref{iv})}}{=} d(M_a, Q)^2+d(M, X)^2-r_M^2+d(I_a, X)^2\\
  & \overset{\text{(\ref{i})}}{=} d(M_a, Q)^2.
\end{align*}
By the Pythagorean theorem $\measuredangle I_a Q_a Q$ is a right angle, where $Q_a$ is the intersection of $C_Q$ and $S_a$ in $D^{n}$. Analogously, we conclude that $\measuredangle I_b Q_b Q$ is a right angle. Hence, the circle $C_Q$ represents a common perpendicular of $S_a$ and $S_b$ and the existence is proven.\par 
To prove the uniqueness we change to the hyperboloid model. 
We assume that $H_a$ and $H_b$ are hyperplanes for the rest of the proof.
According to Definition \ref{Definition: intersecting in a right angle}, a hyperplane $H$ and a geodesic line $L$ intersect in a right angle in $\HH^n$ if there exists $v\in V_L\cap V_H$ such that $\langle\{u_H, v \}\rangle=V_L$, where $u_H$ is the up to sign unique unit vector in $V^\bot_H$.  
From the first part of this proof, we know that $H_a$ and $H_b$ have a common perpendicular. Let $u_{H_a}$ and $u_{H_b}$ be the corresponding unit vectors. The only possibility for a two dimensional subspace that represents a common perpendicular is $\langle \{u_{H_a},u_{H_b}\}\rangle$, since we assume the hyperplanes and their orthogonal complements to be different.
It follows that the existing common perpendicular is $\langle \{u_{H_a},u_{H_b}\}\rangle\cap \HH^n$ and unique. \par
Let $\tilde{H}$ be a hyperplane in $\mathbb{H}^n$ intersecting $H_a$ and $H_b$ in a right angle. 
We have $\langle \{u_{H_i}\} \cup (V_{H_i}\cap V_{\tilde{H}})\rangle = V_{\tilde{H}}$ for $i=a,b$. This implies $u_{H_i}\in V_{\tilde{H}}$ for $i=a,b$. It follows that every hyperplane orthogonal to two ultra-parallel hyperplanes contains their unique common perpendicular. 

\end{proof}

\begin{Rem} 
 
The uniqueness of the common perpendicular gives us a notion of distance between hyperplanes in $\overline{\HH}{}^n$. 
\end{Rem}

\begin{Def} \label{Def:Distance between hyperplanes}
Let $H_a$ and $H_b$ be two hyperplanes in $\HH^n$. The distance $\bar{d}(H_a, H_b)$ between these hyperplanes is defined as follows:
\[
\bar{d}(H_a, H_b):=\begin{cases}
d(\rho(a), \rho(b)) & H_a, H_b \text{ ultra-parallel}\\
0 & \text{ else,}
\end{cases}
\]
where $d(\rho(a), \rho(b))$ is the distance between the intersections $\rho(a), \rho(b)$ of $H_a$ and $H_b$ with their unique common perpendicular $\rho$.
\end{Def}




The rest of this section is devoted to the following lemma, which provides a group theoretic meaning for \textit{ultra-parallel} in hyperbolic reflection groups with a polytope as a fundamental domain.

\begin{Lem}\label{Lem: Intersection hyperplanes}
Let $(W,S)$ be a hyperbolic reflection group in $\HH^n$ with Coxeter polyhedron $P$. Let $\{H_1,\dots, H_m\}$ be the hyperplanes corresponding to elements in $S$ and $R$ be the set of reflections. Further, let $H_{r_1}, H_{r_2}$ be two distinct hyperplanes with $r_i\in R$.
\begin{itemize}
\item[(i)] $H_{r_1}, H_{r_2}$ intersect in $\HH^n$ if and only if there exist $s_i,s_j\in S$ with $m_{ij}<\infty$ and $w\in W$ such that 
\[
r_1r_2=w(s_is_j)^k w^{-1}\quad\text{with}\; k\in \mathbb{Z}\setminus \{0\}.
\]
\item[(ii)] If $P$ is a convex polytope, $H_{r_1}, H_{r_2}$ are not ultra-parallel if and only if there exist $s_i,s_j\in S$ and $w\in W$ such that 
\[
r_1r_2=w(s_is_j)^k w^{-1}\quad\text{with}\; k\in \mathbb{Z}\setminus\{0\}.
\]
\end{itemize}

\end{Lem}

\begin{proof}
We prove (i). Assume that $H_{r_1}$ and $H_{r_2}$ intersect in $\HH^n$. 
Hence, there exists $w_1, w_2\in W$ such that $H_{r_1}= w_1 H_i$ and $H_{r_2}= w_2 H_j$ for some $i,j\in \{1,\dots, m\}$ and  $\tilde{w}=w_1^{-1}\cdot w_2 =(s_is_j)^l$ with $l\in \mathbb{Z}$. Since the intersection $H_{r_1}\cap H_{r_2}$ is not empty in $\HH^n$, we have $m_{ij}<\infty$.
It follows that $r_1=w_1 s_i w_1^{-1}$ and $r_2=w_1 \tilde{w} s_x\tilde{w}^{-1}w_1^{-1}$, where $s_x\in \{i,j\}$. Thus, we obtain
\[
r_1r_2= w_1s_i \tilde{w} s_x\tilde{w}^{-1}w_1^{-1}=w_1(s_is_j)^kw_1^{-1}
\]
with $k \in \mathbb{Z}\setminus \{0\}$.

Now, let $r_1r_2=w(s_is_j)^kw^{-1}$ for some $k\in \mathbb{Z}\setminus \{0\}$, $m_{ij}<\infty$ and $w\in W$. Without loss of generality, we may assume additionally $r_1=s\in S$ (conjugating both sides). Multiplying with $r_1$ yields $r_2=s\cdot w(s_is_j)^kw^{-1}$. Both sides have reflection length $1$. It also may be assumed that $k$ is maximal in the sense that the last letters of a word representing $w$ aren't $s_i$ or $s_j$. If $s\notin \{s_i, s_j\}$, we have $l_R(s\cdot w(s_is_j)^kw^{-1})=3$ with Theorem \ref{Thm: Dyer's Theorem}. We also obtain $l_R(s\cdot w(s_is_j)^kw^{-1})=3$ with the same criterion if we assume that $s$ does not commute with $w$.   Hence, we get $s\in \{s_i, s_j\}$ and $s$ commutes with $w$.
Let $v=s_is_j\cdots$ be a dihedral element with $S$-length $k$. We have $r_1 = w s w^{-1}$ and $r_2=s\cdot w(s_is_j)^kw^{-1}=wvs_xv^{-1}w^{-1}$ with $s_x\in \{s_i,s_j\}$. The hyperplanes $H_i$ and $H_j$ are intersecting, because $m_{ij}<\infty$. Thus, the hyperplanes $H_s$ and $H_{vs_xv^{-1}}$ intersect and so do $H_{r_1}$ and $H_{r_2}$.  
The proof of the second statement works analogously except that we ignore the condition $m_{ij}<\infty$. 
\end{proof}

\section{Arbitrary reflection length close to boundary points}
To prove Theorem \ref{Thm: Reflection length n in certain neighbourhood} for arbitrary hyperbolic reflection groups, we state a series of lemmata first.

\begin{Lem}\label{Lem: Existence of  unique sphere that intersects sphere in certain sphere}
Let $\mathbb{S}^n$ be the unit sphere in the Euclidean space $\mathbb{E}^{n+1}$. For every $(n-1)$-dimensional sphere $O$ in $\mathbb{S}^n$ there exists a unique $(n-1)$-dimensional sphere $S_O$ in $\mathbb{E}^{n+1}$ with 
\[
\mathbb{S}^n \cap S_O = O
\]
and $S_O$ intersects  $\mathbb{S}^n$ in a right angle. 
\end{Lem}

\begin{proof}
We give a sketch of a proof and leave the details to the reader.
Given $O\subseteq \mathbb{S}^n$, we obtain $S_O$ by the following construction: \newline
Take $n+1$ disjoint points $p_1,\dots, p_{n+1}$ in $O$ and consider the tangent spaces $T_1, \dots, T_{n+1}$ at this points. These tangent spaces are distinct hyperplanes in $\mathbb{E}^{n+1}$ intersecting in a single point $c$. The point $c$ is the centre of $S_O$ and together with a point in $O$ it determines $S_O$ uniquely.
\end{proof}

We recall that $D^n$ abbreviates the unit ball in $\mathbb{E}^n$.

\begin{Cor}\label{Cor: Inversion on a sphere fixes as D^n set}
The inversion on a sphere in $\widehat{\mathbb{E}}^{n+1}$ intersecting $\mathbb{S}^{n}$ orthogonally fixes $\mathbb{S}^{n}$ and $D^{n+1}$ as sets.
\end{Cor}
\begin{proof}
The inversion on a sphere in $\widehat{\mathbb{E}}^n$ maps spheres to spheres and preserves the Euclidean angle between intersecting spheres (see \cite{Bridson2009}, Proposition 6.5). 
Together with Lemma \ref{Lem: Existence of  unique sphere that intersects sphere in certain sphere} this implies the corollary.
\end{proof}

The principal significance of the following two lemmata is to draw conclusions from conditions (i) and (ii) in Theorem \ref{Thm: Reflection length n in certain neighbourhood}. These are important ingredients for the proof of the theorem.

\begin{Lem}\label{Lem: Hyperplane in every neighbourhood, parallel hyperplanes with common point}
Let $H_1$ and $H_2$ be two parallel hyperplanes in $\HH^n$ with a common point $\xi\in\partial H_1\cap \partial H_2 \subseteq \partial \HH^n$. Let $S=\{s_1, s_2\}\subseteq Iso(\HH^n)$ be the corresponding reflections across $H_1, H_2$, respectively. For every neighbourhood $U$ of $\xi$ 
there exists a reflection $r$ in the Coxeter group $\langle S\rangle\subseteq Iso(\HH^n)$ such that $H_r\subseteq U$.
\end{Lem}

\begin{proof}
The  proof is conducted in the Poincaré ball model. The hyperplanes $H_i$ are represented by spheres $S_i$ in the one-point compactification $\widehat{\mathbb{E}}^n$ that intersect the unit sphere $\mathbb{S}^{n-1}$ orthogonally.  We can assume that these spheres have radii $r_i$, centre $c_i$ and an antipode $a_i$ to $\xi$ (apply an inversion on the other sphere, in case one sphere contains $\infty$).
Let $(\hat{S}_i)_{i\in \NN}$ be the sequence of spheres defined by 
\[
\hat{S}_0:= S_2\text{ and }\hat{S}_i=i_{\hat{S}_{i-1}}(S_1), 
\]
where $i_{\hat{S}_{i-1}}$ is the inversion on the sphere $\hat{S}_{i-1}$. According to Corollary \ref{Cor: Inversion on a sphere fixes as D^n set}, the unit sphere $\mathbb{S}^{n-1}$ is fixed as a set by all inversions $i_{\hat{S}_{i}}$. All spheres $(\hat{S}_i)_{i\in \NN}$ contain $\xi$, because $\xi$ is fixed by each $i_{\hat{S}_{i}}$.
Let $\hat{c}_i$ be the centre and $\hat{r}_i$ be the radius of $\hat{S}_i$. The points $\xi, c_i,, a_i, \hat{c}_i$ are collinear in $\widehat{\mathbb{E}}^n$ for all $i\in \NN$.  This follows from the definition of the inversion on a sphere. We have $d_2(\xi, \hat{c}_i)= \hat{r}_i$.  Given the centre $\hat{c}_{i-1}$ and the radius $\hat{r}_{i-1}$, the following equations hold according to Formula (\ref{Eq: Formula inversion on a sphere}) for an inversion on a sphere: 
\begin{align*}
\hat{c}_i &= \dfrac{1}{2}\cdot (\xi-i_{\hat{S}_{i-1}}(a_1))+i_{\hat{S}_{i-1}}(a_1), \\
\hat{r}_i &= \dfrac{1}{2}\cdot d_2(i_{\hat{S}_{i-1}}(a_1), \xi),\\
i_{\hat{S}_{i-1}}(a_1) &= \dfrac{\hat{r}_{i-1}{}^2}{||a_1 - \hat{c}_{i-1} ||^2}\cdot (a_1 - \hat{c}_{i-1})+ \hat{c}_{i-1}.
\end{align*}
All $i_{\hat{S}_{i}}(a_1)$ lay in the segment $[\xi, \hat{c}_{i-1}]$.  Since the sequence $(d_2(i_{\hat{S}_{i}}(a_1), \xi))_{i\in \NN}$ is monotonously decreasing and bounded by $0$, it converges to $0$, the only possible limit. \par
Hence, by reflecting $S_1$ across $\hat{S}_i$, we obtain spheres with arbitrary small diameter in $\mathbb{E}^n$ that contain $\xi$ as $i\to \infty$.
The spheres represent hyperplanes corresponding to reflections in $\{wsw^{-1}\mid s\in S, w\in \langle S\rangle \}$ and we find a sufficiently small sphere that represents a hyperplane $H_r$ contained in $U$.
\end{proof}

\begin{Lem}\label{Lem: Ultra parallel hyperplanes, hyperplane in neighbourhood of endpoint of perpendicular}
Let $H_1$ and $H_2$ be two ultra-parallel hyperplanes in $\HH^n$ and $S=\{s_1, s_2\}\subseteq Iso(\HH^n)$ be the corresponding reflections. Let $\gamma$ be a geodesic ray contained in the unique common perpendicular $\bar{\gamma}$ of $H_1$ and $H_2$.  For every neighbourhood $U$ of $\gamma(\infty)$ there exists a reflection $r$ in the Coxeter group $\langle S\rangle\subseteq Iso(\HH^n)$ such that $H_r\subseteq U$.
\end{Lem}

\begin{proof}
Isometries preserve angles and intersection in a right angle of subspaces. Let $R$ be the set of reflections in $\langle S\rangle$. From Formula (\ref{Eq: Formular reflection on hyperplane}) in Section \ref{section 1} for a reflection on a hyperplane in the hyperboloid model, it is evident that $\bar{\gamma}$ gets fixed as a set by reflecting on $H_{\tilde{r}}$ for 
all $\tilde{r}$ in $R$. Thus, all hyperplanes in $\mathcal{H}= \{H_{\tilde{r}}\mid \tilde{r}\in R \}$  have $\bar{\gamma}$ as a common  perpendicular.\par
Just as in the proof of Lemma \ref{Lem: Hyperplane in every neighbourhood, parallel hyperplanes with common point}, we continue in the Poincaré ball model and define a sequence of spheres. Let $S_i$ be the sphere in $\widehat{\mathbb{E}}^n$ representing $H_i$. Without loss of generality, we can assume that $S_i$ has radius $r_i$, centre $c_i$, intersection $x_i$ with the segment $[c_1, c_2]$ in $\mathbb{E}^n$ and antipode $a_i$ to $x_i$ (apply an inversion on the other sphere, in case one sphere contains $\infty$).
Let $(\hat{S}_i)_{i\in \NN}$ be the sequence of spheres defined by 
\[
\hat{S}_0:= S_2\text{ and }\hat{S}_i=i_{\hat{S}_{i-1}}(S_1), 
\]
where $i_{\hat{S}_{i-1}}$ is the inversion on the sphere $\hat{S}_{i-1}$.
For all $\hat{S}_i$ exists a reflection $r_i\in R$ such that $\hat{S}_i$ represents $H_{r_i}$.
Let $\hat{c}_i$ be the centre of $\hat{S}_i$, $\hat{x}_i$ be the intersection of $\hat{S}_i$ with the segment $[c_1, c_2]$ in $\mathbb{E}^n$ and $\hat{a}_i$ be the antipode of $\hat{x}_i$ in $\hat{S}_i$.
From the proof of Theorem \ref{Thm: Ultra-parallel theorem for subspaces} and from the Formula \ref{Eq: Formula inversion on a sphere} in section \ref{section 1}, we deduce that all $\hat{c}_i,\hat{x}_i, \hat{a}_i$ and $\gamma(\infty)$ lay on the geodesic segment $[c_1, c_{2}]$.
Given the points $\hat{c}_{i-1}, \hat{a}_{i-1}, \hat{x}_{i-1}$,  Formula (\ref{Eq: Formula inversion on a sphere}) implies: 
\begin{align*}
\hat{x}_i &= i_{\hat{S}_{i-1}}(x_{1}) =\dfrac{d_2(\hat{a}_{i-1},\hat{x}_{i-1})^2}{4\cdot ||x_{1} - \hat{c}_{i-1} ||^2}\cdot (x_{1} - \hat{c}_{i-1})+ \hat{c}_{i-1}\quad\in (\hat{x}_{i-1}, \gamma(\infty)),\\
\hat{a}_i &= i_{\hat{S}_{i-1}}(a_{1}) =\dfrac{d_2(\hat{a}_{i-1},\hat{x}_{i-1})^2}{4\cdot ||a_{1} - \hat{c}_{i-1} ||^2}\cdot (a_{1} - \hat{c}_{i-1})+ \hat{c}_{i-1}\quad \in (\gamma(\infty), \hat{c}_{i-1}) ,\\
\hat{c}_i &= \frac{1}{2}(\hat{x}_i-\hat{a}_i)+\hat{a}_i \quad \in (\gamma(\infty),\hat{a}_{i} ) .
\end{align*}
We consider the sequence of distances $(d_2(\hat{c}_i, \hat{x}_i))_{i\in \NN}$, where $\hat{c}_i$ is outside the closed unit ball $\overline{D}{}^n$ and $\hat{x}_i$ is always inside the open unit ball $D{}^n$, since the inversion on a sphere intersecting $\mathbb{S}^{n-1}$ orthogonally is a bijection on $\widehat{\mathbb{E}}^n$ and fixes $\mathbb{S}^{n-1}$ as well as $D^n$ as sets (Corollary \ref{Cor: Inversion on a sphere fixes as D^n set}). This sequence is monotonously decreasing and converges to $0$, since the sequences of points $(\hat{c}_i)_{i\in \NN}$ and $(\hat{x}_i)_{i\in \NN}$ both converge to $\gamma(\infty)$. Hence, the sequence of spheres $(\hat{S}_i)_{i\in \NN}$ contains elements with arbitrary small radius and centre arbitrarily close to $\gamma(\infty)$ in $\mathbb{E}^n$.  This implies the existence of a hyperplane $H_{r}$ in the sequence $(H_i)_{i\in \NN}$ of hyperplanes corresponding to $(\hat{S}_i)_{i\in \mathbb{N}}$ such that $H_{r}\subseteq U$ with $r\in R$.
\end{proof}

We restate Theorem \ref{Thm: Reflection length n in certain neighbourhood}.

\setcounter{Theorem}{0}
\begin{Theorem}
Let $(W,S)$ be a hyperbolic reflection group in $\HH^n$ with fundamental domain $P$. Let $R$ be the set of reflections in $W$. Let $U$ be a neighbourhood in $\overline{\HH}{}^n$ of a point $\xi$ in $\partial \HH^n$. Suppose $\xi$ satisfies one of the following conditions:
\begin{itemize}
\item[(i)] $\xi$ is a common point of two parallel hyperplanes $H_r, H_{r'}$ with $r,r'\in R$.
\item[(ii)]$\xi$ is an endpoint of the common perpendicular of two ultra-parallel hyperplanes $H_r, H_{r'}$ with $r,r'\in R$.
\end{itemize}
For every $k\in \NN$ there exists $w\in W$ with $l_R(w)=k$ such that the domain $wP$ is contained in $U$. \com{If $P$ is not a polytope the other direction is also true?, maybe you can proof duszenko with this}
\end{Theorem}

\begin{proof}
For the point $\xi\in \partial \HH^n$, let $U(c,d,\epsilon)$ be contained in the neighbourhood $U$, where $c$ is a geodesic ray with $c(\infty)= \xi$. In both cases there exists a hyperplane $H_r$ with $r\in R$ completely contained in $U(c,d,\epsilon)$ by Lemma \ref{Lem: Hyperplane in every neighbourhood, parallel hyperplanes with common point} and Lemma \ref{Lem: Ultra parallel hyperplanes, hyperplane in neighbourhood of endpoint of perpendicular}.
So the image $rP$ of $P$ under $r$ is contained in $U$. This proves the theorem for reflection length $1$. 
In general, $l_R$ is unbounded on $W$ (see Theorem \ref{Thm: Duszenko LR unbpunded}).  Thus, there exists $w\in W$ with reflection length $l_R(w)=k+1$.
In case $wP$ is in the half-space of $H_r^-$ that is contained in $U$, the proof is complete. Else $wP\subseteq H_r^+$ and we reflect $wP$ across $H_r$ to obtain $wrP\subseteq U$.
The element $wr$ has reflection length $n+1\pm 1$ (see \cite{Lewis2018} , 1.3).
A sequence $(w_1P, \dots, w_lP)$ of adjacent tiles in $U$ between $rP$ and $wrP$ contains tiles $w_iP$ of all reflection lengths $l_R(w_i)$ between $1$ and $l_R(wr)$ (see \cite{Lewis2018}, Paragraph 1.3). 
\end{proof}

\section{Hyperplanes not generating a hyperbolic reflection group}\label{Section 4}
In this section, we state three results that we need in the last section but hold in a more general setting. 

\begin{Lem}\label{Lem: Hyperebenen mit gemeinsamen Punkt in Rand erzeugen Dinfty}
For $k\geq 3$ let  $\{H_{r_1}, \dots, H_{r_k}\}$ be a set of pairwise parallel hyperplanes in $\HH^n$. If all $H_{r_i}$ have a common point $\xi$ in $\partial\HH^n$ and the group $D$ generated by the reflections $r_i$ across $H_{r_i}$ is a discrete subgroup of $Iso(\HH^n)$, then $D$ is isomorphic as a group to the infinite dihedral group $D_\infty$. 
\end{Lem}

\begin{proof}
The uniqueness of $\xi\in \partial \HH^n$ as a common point of $\partial H_{r_1}, \dots, \partial H_{r_k}$ follows  from  Corollary \ref{Lem: Intersection of Hyperplanes empty or point}.
The hyperplane reflection $r_i$ across $H_{r_i}$ maps hyperplanes in $\{H_{r_1}, \dots, H_{r_k}\}$ to hyperplanes with $\xi$ in their boundary, because $H_{r_i}$ and $\partial H_{r_i}$ are fixed pointwise by $r_i$. 
Let $R$ be the set of all reflections in the group $D=\langle r_i\mid i\in \{1,\dots, k\}\rangle$. The corresponding hyperplanes all contain $\xi$ in their boundary.\par
We fix a hyperplane $H\in \mathcal{H}:= \{H_r\mid r\in R\}$. 
Since $W$ is discrete, we can choose $\varepsilon, \delta\in \{+,-\}$ such that the intersection of half-spaces $H^\varepsilon\cap H_1^\delta$ is not empty nor a half-space and contains no other hyperplane in $\mathcal{H}$.
We show that the reflections $s_1$ and $s$ corresponding to $H_1$ and $H$ generate $W$. Therefore, we assume that there exists a $H_{r_i}$ such that $r_i$ can't be written as a product of $s_1$ and $s$. There exists an element $w \in  \langle\{s, s_1\}\rangle$ such that
\[
H_{r_i}\subseteq wH^\varepsilon\cap wH_1^\delta.
\]
In words, the hyperplane $H_{r_i}$ is contained in the intersection of two half-spaces of copies of $H$ and $H_1$, because $H_{r_i}$ contains $\xi$ in its boundary, too.
Thus, the hyperplane $w^{-1}H_{r_i}\in \mathcal{H}$ is contained in  $H^\varepsilon\cap H_1^\delta$. This contradicts our assumption and we proved  $\langle\{s, s_1\}\rangle= D$. Since  $H$ and $H_1$ don't intersect in $\HH^n$, $ss_1$ has infinite order and $D$ is isomorphic to $D_\infty$. 
\end{proof}

\begin{Lem}\label{Lem: Common perpendicular generates Dinfty}
Let $\{H_{r_1}, \dots, H_{r_k}\}$ be a set of pairwise ultra-parallel hyperplanes in $\HH^n$ with $k\geq 3$ and a common perpendicular $\varrho$. If the group $D$ generated by the reflections $r_i$ across $H_{r_i}$ is a discrete subgroup of $Iso(\HH^n)$, then $D$ is isomorphic as a group to $D_\infty$. 
\end{Lem}
\begin{proof}

The hyperplane reflection $r_i$ across $H_{r_i}$ maps hyperplanes in $\{H_{r_1}, \dots, H_{r_k}\}$ to pairwise ultra-parallel hyperplanes with $\varrho$ as their common perpendicular, because $r_i$ fixes $\varrho$ as a set. Let $R$ be the set of all reflections in the group $D=\langle r_i\mid i\in \{1,\dots, k\}\rangle$. The hyperplanes in the set $\mathcal{H}:= \{ H_r\mid r\in R\}$ are all pairwise ultra-parallel and have $\varrho$ as their common perpendicular.
The rest of the proof is analogous to the one of Lemma \ref{Lem: Hyperebenen mit gemeinsamen Punkt in Rand erzeugen Dinfty}.
\end{proof}

\begin{Lem}\label{Lemma: Arbitrary distance between hyperplanes}
Let $(W,S)$ be a hyperbolic reflection group with fundamental polyhedron $P$. Let $R$ be the set of reflections. For every 
$r\in R$  and ever $\epsilon>0$ there exists $r_\epsilon\in R$ such that the hyperplanes $H_{r_\epsilon}$ and $H_r$ are ultra-parallel and 
\[
\bar{d}(H_r, H_{r_\epsilon})>\epsilon.
\]
\end{Lem}

\begin{proof}
It suffices to show that there exists $\tilde{r}\in R$ with $\bar{d}(H_r, H_{\tilde{r}})>0$, since we can reflect these hyperplanes several times on each other to obtain the $H_{r_\epsilon}$. Hyperplane reflections fix perpendiculars as sets. \par
$W$ has universal Coxeter groups of arbitrary large rank as  reflection subgroups (see \cite{Edgar2013}). Let $W_3= \langle r_1, r_2, r_3 \rangle$ be a universal Coxeter group contained in $W$ as a reflection subgroup with $\{r_1, r_2, r_3\}\subseteq R$ as a Coxeter generating set. $H_{r_1}, H_{r_2}$ and $H_{r_3}$ do not intersect pairwise in $\HH^n$, since $r_1, r_2, r_3$ have pairwise infinite order. Suppose that all $H_{r_i}$ intersect with $H_r$ in $\overline{\HH}{}^n$. Else there would be a hyperplane with distance greater than zero to $H_r$.\par
Assume that the hyperplanes $H_{r_1}$ and $H_r$ do not intersect in $\HH^n$ and have a common point in $\partial\HH^n$.
$H_r$ is contained in one half-space $H_{r_1}^\varepsilon$ with $\varepsilon\in \{+,-\}$. Each of $H_{r_2}$ and $H_{r_3}$ is also contained in a half-space.
Since $W_3$ is not isomorphic to $D_\infty$, the hyperplanes $H_{r_1}, H_{r_2}$ and $H_{r_3}$ have no common point in $\overline{\HH}{}^n$ (see Lemma \ref{Lem: Hyperebenen mit gemeinsamen Punkt in Rand erzeugen Dinfty}).
This implies that there exists $H_{r_j}$ with $j\in \{2,3\}$ such that $H_{r_j}$ doesn't intersect with $H_r$ in the point $H_r\cap H_{r_1}\subseteq \overline{\HH}{}^n$ (see Lemma \ref{Lem: Intersection of Hyperplanes empty or point}). Hence, $H_{r_j}$ is contained in $H_{r_1}^\varepsilon$, too. 
The hyperplane's reflection $H_{r_1r_jr_1}$ across $H_{r_1}$ does not intersect with $H_r$ in $\overline{\HH}{}^n$. \par
Let us assume that $H_r$ intersects all hyperplanes $H_{r_i}$ in $\HH^n$. Since $\langle\{r_1, r_2, r_3\}\rangle$ is isomorphic to a universal Coxeter group, the intersection of two distinct non-ultra-parallel hyperplanes $H_{r_i}$ and $H_{r_j}$, $1\leq i,j\leq 3$, is empty and $\partial H_{r_i}\cap \partial H_{r_j}= \{\delta\}$ (see Corollary \ref{Lem: Intersection of Hyperplanes empty or point}). There exists a neighbourhood $U$ of $\delta$ in $\overline{\HH}{}^n$ such that $U\cap  \overline{H}_r = \emptyset$, because $H_r$ intersects all $H_{r_i}$. By Lemma \ref{Lem: Hyperplane in every neighbourhood, parallel hyperplanes with common point} , we obtain hyperplanes in arbitrary small neighbourhoods of $\delta\in \partial \HH^n$ in $\overline{\HH}{}^n$. This gives us $H_{\tilde{r}}\subseteq U$ with $\tilde{r}\in R$ and $H_{\tilde{r}}$ is ultra-parallel to $H_r$.
Let the $H_{r_i}$ be pairwise ultra-parallel.  According to Lemma \ref{Lem: Common perpendicular generates Dinfty}, we can assume that $H_r$ is not orthogonal to $H_{r_i}$. This implies that $H_r$ does not contain the unique common perpendicular $\gamma$ of $H_{r_i}$ and $H_{r_j}$ (see Theorem \ref{Thm: Ultra-parallel theorem for subspaces}). Hence, by Lemma \ref{Lem: Ultra parallel hyperplanes, hyperplane in neighbourhood of endpoint of perpendicular} there exists a hyperplane $H_{\hat{r}}$ with $\hat{r}\in R$ in a neighbourhood $V$ of an endpoint $\gamma(\infty)$ such that $V\cap H_r= \emptyset$. $H_r$ and $H_{\hat{r}}$ are ultra-parallel.
\end{proof}

\section{Polytopes as fundamental domains}

\begin{Lem}\label{Lem: dense endpoints}
Let $(W,S)$ be a hyperbolic reflection group with a convex polytope $P$ as a fundamental domain. Let $R$ be the set of reflections in $W$. The set of ideal points of hyperplanes \[\mathcal{H}_\infty:=\{\gamma(\infty)\;|\; \gamma\subseteq H_r \text{ geodesic ray, } r\in R\}\] is dense in $\partial \mathbb{H}^n$.
\end{Lem}

\begin{proof}
Since $W$ is discrete and the set $\mathcal{H}_\infty$ is a countable union of spheres homeomorphic to $\mathbb{S}^{n-2}$, we have 
$
\mathcal{H}_\infty\subsetneq \partial\mathbb{H}^n
$
(see Section \ref{Subsection: Hyperbolic hyperplanes and their boundary}). 
Let $\xi\in \partial\mathbb{H}^n \setminus \mathcal{H}_\infty$ and let $c:[0,\infty]\to \overline{\mathbb{H}}{}^n$ be a geodesic ray with $c(\infty)=\xi$.
For $\epsilon, r>0$ the sets 
$
U(c,s,\epsilon)
$
form a neighbourhood basis of $\xi$ (see Definition \ref{Def: cone topology}).
We fix arbitrary $\epsilon, r>0$. In the Poincaré ball model, the set of the endpoints $C(\infty)$ of geodesic rays in  
\[
C:=\{ c^*\;|\; c^*\text{ geodesic ray, }c^*(0)=c(0),\; d(c^*(r), c(r))=\epsilon \}
\] 
is a sphere in $\partial \HH^n$.
According to Lemma \ref{Lem: Existence of  unique sphere that intersects sphere in certain sphere}, there exists a hyperplane $E$, possibly not in $\{H_r\;|\; r\in R\}$, with 
$
\partial E= C(\infty).
$
Let $E^+$ be the half-space such that $\xi$ is in the closure $\overline{E^+}$.
To prove the lemma, it is sufficient to show that there exists a hyperplane $H_{r'}$ with $r'\in R$ that intersects the half-space $E^+$ non-empty.
Then $U(c,s,\epsilon)\cap \partial H_r'\neq \emptyset$, because hyperplanes as well as $\mathbb{H}^n$ are uniquely geodesic subspaces (see \cite{Bridson2009}, p. 21).
Since $P$ is a strict fundamental domain, there exists $w\in W$ with $E^+\cap wP\neq\emptyset$. 
This implies that there exists a hyperplane $H_{r'}$ with $r'\in R$ such that $E^+\cap H_{r'}\neq \emptyset$, because $P$ is a convex polytope. 
\end{proof}

\begin{Not}\label{Not: Sets of ideal points}
Let $(W,S)$ be a hyperbolic reflection group in $\mathbb{H}^n$. Let $R$ be the set of reflections. We write $\mathcal{I}_p(W)$ for the set of ideal points $\xi\in \partial \HH^n$, such that there exists two parallel hyperplanes $H_{r_1}$ and $H_{r_2}$ corresponding to $r_1, r_2\in R$ with $\xi\in \partial H_{r_1}\cap \partial H_{r_2}$.
Furthermore, we write $\mathcal{P}_{up}(W)$ for the set of ideal points $\mu\in \partial \HH^n$, such that $\mu$ is an endpoint of the common perpendicular of two ultra-parallel hyperplanes $H_{r_3}$ and $H_{r_4}$ corresponding to $r_3, r_4\in R$. 
\end{Not}

\begin{Thm}\label{Thm: dense endpoints}
Let $(W,S)$ be a hyperbolic reflection group in $\mathbb{H}^n$ with a convex polytope $P$ as a fundamental domain. The union  
\[
\mathcal{I}_p(W)\cup \mathcal{P}_{up}(W)
\] 
is dense in the visual boundary $\partial \mathbb{H}^n$.
\end{Thm}

\begin{proof}
Assume that there exists a point $\xi\in \partial\HH^n$, real numbers $\epsilon, \delta>0$ and a neighbourhood $U(c,\delta,\epsilon)$ of $\xi$ such that $\mathcal{I}_p(W)\cap U(c,r,\epsilon)=\emptyset$. Let $R$ be the set of reflections. We want to show $\mathcal{P}_{up}(W)\cap U(c,\delta,\epsilon)\neq \emptyset$ and begin by proving that there exists an endpoint $\varphi(\infty)\in U(c,\delta,\epsilon)$ of a geodesic ray $\varphi$ contained in a one-dimensional intersection of finitely many hyperplanes in $\mathcal{H}=\{H_r\mid r\in R\}$. Therefore, we successively follow the implications for each dimension lower hereafter:\par
Lemma \ref{Lem: dense endpoints} states that the endpoints $\mathcal{H}_\infty=\{\gamma(\infty)\;|\; \gamma\subseteq H_r \text{ geodesic ray, } r\in R\}$ are dense in $\partial\HH^n$. 
Hence, there exists a hyperplane $H_r$ with $r\in R$ and an ideal point $\nu$ such that $\nu\in \partial H_r\cap U(c,\delta,\epsilon)$. The hyperplane $H_r$ is isometric to $\HH^{n-1}$ (see Section \ref{Subsection: Hyperbolic hyperplanes and their boundary}) and the intersection $\partial H_r\cap U(c,\delta,\epsilon)$ is a neighbourhood of $\nu$ in $\overline{H_r}$. We consider hyperplanes in $H_r$ that are intersections $H_r\cap H_t$ for $t\in R$. There exists $w\in W$ such that $P'=wP\cap H_r$ is a polytope in $H_r$ and we can apply Lemma \ref{Lem: dense endpoints} again as well as analogous arguments in one dimension lower until dimension one.\par
Thus, there exists a geodesic line $\gamma$ in $\HH^n$ that is the intersection of finitely many hyperplanes in $\mathcal{H}$ and has an endpoint $\gamma(\infty)$ in $U(c,\delta,\epsilon)$. The geodesic ray ${\gamma}$  contains dimension-$1$ faces of infinitely many $uP$ with $u\in W$, since $P$ is a polytope and we assume that $\mathcal{I}_p(W)\cap U(c,\delta,\epsilon)$ is empty.
Considering the dimension-$0$ faces of the $uP$, this implies that $\gamma$ is intersected punctually in intervals of finitely many different lengths by hyperplanes $H_t$ with $t\in T\subseteq R$ with $|T|=\infty$.\par 
In the Poincaré ball model, this setting translates to spheres $S_t$ with $t\in T$ intersecting the circle $C_\gamma$ corresponding to $\gamma$, where $S_t$ and $C_\gamma$ intersect $\mathbb{S}^{n-1}$ orthogonally for all $t\in T$. The spheres $S_t$ intersect $C_\gamma$ in a finite number of different angles, because the only angles that can occur are inherited from $P$ as it is a convex polytope and a strict fundamental domain. Thus, we can choose a sphere $S_{t'}$ with $t'\in T$ intersecting $C_\gamma$ sufficiently close to $\gamma(\infty)$, such that the intersection $S_{t'}\cap D^n$ with the open unit ball $D^n$ is contained in $U(c,\delta,\epsilon)$. One half-space $H_{t'}^\varepsilon$ of the hyperplane corresponding to $S_{t'}$ is contained in $U(c,\delta,\epsilon)$.  For $H_{t'}$ exists a hyperplane $H_{\tilde{t}}\in \mathcal{H}$ that is ultra-parallel to $H_{t'}$ (see Lemma \ref{Lemma: Arbitrary distance between hyperplanes}). The unique common perpendicular $\varrho$ of $H_t$ and $H_{t'}$, which exists by Theorem \ref{Thm: Ultra-parallel theorem for subspaces}, has one endpoint $\varrho(\infty)\in \mathcal{P}_{up}(W)$ in the boundary $\partial H_{t'}^\varepsilon\subseteq U(c,\delta,\epsilon)$. It follows $\mathcal{P}_{up}(W)\cap  U(c,\delta,\epsilon)\neq \emptyset$, which completes the proof.
\end{proof}

\begin{Thm}\label{Thm: Reflection length n in every neighbourhood}
Let $(W,S)$ be a hyperbolic reflection group in $\HH^n$ with a polytope $P$ as a fundamental domain. Let $R$ be the set of reflections in $W$. Let $U$ be a neighbourhood in $\overline{\HH}{}^n$ of a point $\xi$ in $\partial \HH^n$.
For every $k\in \NN$ there exists $w\in W$ with $l_R(w)=k$ such that the domain $wP$ is contained in $U$.
\end{Thm}

\begin{proof}
 Theorem \ref{Thm: dense endpoints} states that there exists a point $\nu\in \partial \HH^n$ in $U$ such that  $\nu$ is either an ideal point of two parallel hyperplanes corresponding to reflections in $R$ or one endpoint of a common perpendicular of two ultra-parallel hyperplanes corresponding to reflections in $R$. So $\nu$ satisfies one of the conditions in Theorem \ref{Thm: Reflection length n in certain neighbourhood} and in every neighbourhood of $\nu$ for every $k\in\NN$ exists $wP$ with $l_R(w) = k$. Since $\HH^n$ is a metric space, there exists a neighbourhood $U'$ of $\nu$ that is completely contained in $U$ and the theorem is proven.
\end{proof}

\begin{Rem}
Theorem \ref{Thm: Reflection length n in every neighbourhood} is only true for reflection groups with a polytope as a fundamental domain. For reflection groups with polyhedra that aren't polytopes as fundamental domains, the theorem is false. One example of such a reflection group is the hyperbolic reflection group generated by three pairwise ultra-parallel hyperplanes in $\HH^n$. This group is isomorphic to the universal Coxeter group $W_n$ with three generators. Let $w$ be in $W_n$ and $R$ be the set of reflections in $W_n$. Since there exist neighbourhoods $U_w$ of a point $\xi_w$ in the boundary of every domain $wP$ such that $U_w\cap wP=U_w$, the only reflection length arbitrary close to $w\xi$ is the reflection length $l_R(w)$. 
\end{Rem}

\paragraph{\bf{Acknowledgements}}The author thanks Petra Schwer and Thomas Kahle for many helpful suggestions 
during the preparation of the paper. Further, the author thanks Sami Douba for drawing the author’s attention to the work of Felikson and Tumarkin. The author was supported partially by the
Deutsche Forschungsgemeinschaft (DFG, German Research Foundation) – 314838170, GRK 2297
MathCoRe.

\bibliographystyle{amsrefs}
\bibliography{Literatur.bib}

\end{document}